\documentclass{asl}

\newtheoremstyle{mystyle}{}{}{\slshape}{2pt}{\scshape}{.}{ }{} 

\newtheorem{thm}{Theorem}[section]
\newtheorem{cor}[thm]{Corollary}

\newtheorem{prop}[thm]{Proposition}
\newtheorem{lemme}[thm]{Lemma}
\newtheorem{fait}[thm]{Fact}

\newtheorem{conjecture}[thm]{Conjecture}

\newtheorem{problem}[thm]{Problem}

\newtheorem{obs}[thm]{Observation}
\theoremstyle{definition}
\newtheorem{defi}[thm]{Definition}
\theoremstyle{mystyle}

\theoremstyle{remark}
\newtheorem{rem}[thm]{Remark}

\newcommand{\impl}{\rightarrow}
\newcommand{\pprec}{\prec^+}
\newcommand{\monster}{\mathcal U}

\newcommand{\ordI}{\mathcal I}
\newcommand{\ordJ}{\mathcal J}

\DeclareMathOperator{\tp}{tp}

\DeclareMathOperator{\acl}{acl}
\DeclareMathOperator{\dpr}{dp-rk}
\DeclareMathOperator{\dprk}{dp-rk}

\DeclareMathOperator{\alt}{alt}

\title{Dp-minimality: invariant types and dp-rank}
\author{Pierre Simon}
\revauthor{Simon, Pierre}
\address{Universit\'e de Lyon; CNRS\\
Universit\'e Lyon 1\\
Institut Camille Jordan UMR5208\\
43 boulevard du 11 novembre 1918\\
69622 Villeurbanne Cedex, France}

\thanks{Partially supported by the European Research Council under the European Unions Seventh Framework Programme (FP7/2007-2013) / ERC Grant agreement no. 291111.}
\thanks{Partially supported by ValCoMo (ANR-13-BS01-0006)}

\begin{document}
\begin{abstract}This paper has two parts. In the first one, we prove that an invariant dp-minimal type is either finitely satisfiable or definable. We also prove that a definable version of the (p,q)-theorem holds in dp-minimal theories of small or medium directionality.

In the second part, we study dp-rank in dp-minimal theories and show that it enjoys many nice properties. It is continuous, definable in families and it can be characterised geometrically with no mention of indiscernible sequences. In particular, if the structure expands a divisible ordered abelian group, then dp-rank coincides with the dimension coming from the order.
\end{abstract}
\maketitle


%

The class of dp-minimal theories is a generalisation suggested by Shelah of the classes of o-minimal and C-minimal theories. It also contains the field $\mathbb Q_p$ of p-adics. Strongly related to it is the notion of \emph{dp-rank} defined in NIP theories as follows: the dp-rank of a partial type $\pi(x)$ over $A$ is $\geq k$ if there are $a\models \pi$ and $k$ sequences ($I_i$, $i<k$) mutually indiscernible over $A$ (that is, $I_i$ is indiscernible over $AI_{\neq i}$) none of which is indiscernible over $Aa$. Dp-minimal theories are theories in which all 1-types have dp-rank 1, or equivalently $\dpr(x=x)=1$.

This paper is divided into two main sections which can be read independently. In the first one we study invariant types and a definable version of the $(p,q)$-theorem. In the second one, we prove some properties of dp-rank in dp-minimal theories.

We first present our results on invariant types. Assume that $T$ is NIP. In \cite{Sh900}, Shelah proves that given an arbitrary type $p$ over some saturated $M$ and $a\models p$, one can find some tuple $c\in \monster$ such that $\tp(c/M)$ is finitely satisfiable in a small $B\subseteq M$ and $\tp(a/cM)$ is weakly orthogonal to all types finitely satisfiable in some small $B'\subseteq M$. Thus one can consider $c$ as a maximal analysis of $p$ over finitely satisfiable types. Our initial idea was to consider what happens when $p$ is taken to be invariant. It is an easy observation that if $p$ is orthogonal to all finitely satisfiable types ({\it i.e.}, we cannot do any analysis), then it must be definable (see Lemma \ref{lem_def}). This gave rise to the hope of being able to analyse any invariant type over finitely satisfiable types, with a definable `quotient'. However, this ended up being harder than expected and the main questions are left unresolved. We only manage to treat the dimension one case, where no mixed situation can occur. Thus our first main theorem is the following:

\begin{thm}
If $p$ is an invariant type of dp-rank 1, then it is either finitely satisfiable or definable.
\end{thm}

In turned out that those ideas were useful in studying another related problem: that of finding definable (or partially definable) types. If $p(x)$ is a global $M$-invariant type, then for any formula $\phi(x;y)$, we can consider the subset $d_p\phi \subseteq S_y(M)$ of types $q(y)$ such that $p\vdash \phi(x;b)$ for any $b\models q$. The type $p$ is definable exactly when the sets $d_p\phi$ are open for all $\phi$ (and hence they are also closed). We are concerned here with finding types extending some given formula and for which some prescribed type $q$ falls in the interior of $d_p\phi$. We only succeed under strong assumptions on the theory.

\begin{thm}
Assume that $T$ is dp-minimal and of medium or small directionality, then given any model $M$ and formula $\phi(x;b)\in L(\monster)$, if $\phi(x;b)$ does not fork over $M$, then there is a formula $\theta(y)\in \tp(b/M)$ such that $\bigwedge_{b'\in \theta(\monster)} \phi(x;b')$ is consistent (and hence does not fork over $M$).
\end{thm}

That this holds in any NIP theory was conjectured in \cite{ExtDef2}. This conjecture amounts to asking for a definable version of the $(p,q)$-theorem of finite combinatorics, as we will explain in Section \ref{sec_invtypes}.

In a subsequent paper \cite{SimStar} with Sergei Starchenko, we show that one can adapt the constructions given here to find definable types in dp-minimal theories with definable Skolem functions. With those hypothesis, we show the existence of a definable type extending any non-forking formula. In particular, this holds for $\mathbb Q_p$.

The second part of this paper studies dp-rank in dp-minimal theories. Little is known about dp-rank in general, apart from the fact that it is sub-additive (\cite{KOU}). This implies that the dp-rank of an $n$ tuple in a dp-minimal theory has rank at most $n$. For that reason, we only work with finite ranks here, whereas in general the dp-rank can be an infinite cardinal (see e.g., \cite[Chapter 4]{NIPbook}). In \cite{witnessdp} we proved with Itay Kaplan that---after extending the base---the sequences $I_i$ in the definition of dp-rank can be taken to be sequences of realisations of $p$. The first main result of this part (Proposition \ref{prop_mainlemma}) is a strengthening of this for dp-minimal theories. In fact, the situation is as good as it could possibly be: the dp-rank of a tuple $(a_1,\ldots,a_n)$ can be witnessed by mutually indiscernible sequences of points, each of them starting with one of the $a_i$'s. As an immediate consequence, the property ``$\dpr(\bar a/A)\geq k$" is type-definable in $\bar a$.

Our second main result (Theorem \ref{th_rectangle}) says that dp-rank can be characterised without mentioning indiscernible sequences and implies, if $T$ eliminates $\exists^{\infty}$, that it is definable in families. Our theorems can be summarised as follows:

\begin{thm}\label{th_summary}
Let $T$ be dp-minimal, and for $\oplus_1$ assume elimination of $\exists^{\infty}$. We work only in real sorts.

$\oplus_0$ If $\acl$ satisfies exchange, then dp-rank coincides with $\acl$-dimension and this happens if and only if dp-rank is additive.

$\oplus_1$ For every formula $\phi(\bar x;\bar y)$ and integer $k$, the set of parameters $\bar b$ for which $\dpr(\phi(\bar x;\bar b))=k$ is definable.

$\oplus_2$ Let $\bar a$ be a tuple and $A$ a set of parameters. Then there is a formula $\phi(\bar x)\in \tp(\bar a/A)$ such that $\dprk(\phi(\bar x))=\dprk(\bar a/A)$.

$\oplus_3$ The formula $\phi(x_0,\ldots,x_{n-1})$ has dp-rank $n$ if and only if there are formulas $\theta_k(x_k)$ of dp-rank 1 and a formula $\psi(x_0,\ldots,x_{n-1})$ of dp-rank $<n$ such that $\phi(\bar x)$ contains the definable set $\bigwedge_{k<n} \theta_k(x_k) \setminus \psi(\bar x)$. In fact, $\psi(\bar x)$ can even be taken as a \emph{hypersurface} (meaning that when we project on the first variable, all fibers have dp-rank $<n-1$).

$\oplus_{3'}$ Assume that $T$ has no non-realised generically stable type. Then $\oplus_3$ holds with $\psi(\bar x)=\bot$.

$\oplus_4$ The formula $\phi(x_0,\ldots,x_{n-1})$ has dp-rank $\geq k$ if and only if its projection to some $k$ variables has dp-rank $k$.
\end{thm}

Note the following consequence: in dp-minimal theories, one can give an alternative, equivalent definition of dp-rank as follows. A formula has dp-rank at least $1$ if and only if it is infinite. Then inductively, a formula $\phi(x_0,\ldots,x_{n-1})$ in $n$ free variables has dp-rank $n$ if and only if $\oplus_3$ is satisfied. Finally using $\oplus_4$, a general formula $\phi(x_0,\ldots,x_{n-1})$ has dp-rank $k$, where $k\leq n$ is maximal such the projection of $\phi$ to some $k$ variables has dp-rank $k$.

Elimination of $\exists^{\infty}$ is necessary for $\oplus_1$ to hold. Without it, we still can show a weaker statement: the set of parameters $\bar b$ such that $\dpr(\phi(\bar x;\bar b))> k$ is type-definable.

It is proved in \cite{dpmin} that in a dp-minimal divisible ordered group, any infinite definable set in dimension 1 has non-empty interior. We then deduce from $\oplus_{3'}$ that under the same hypothesis, any definable set in $n$ variables has dp-rank $n$ if and only if it has non-empty interior. We will in fact prove this directly in Section \ref{sec_linear}.

Finally, notice that the hypersurface $\psi(\bar x)$ in $\oplus_3$ is necessary in general. For example, in the theory of pure equality the formula $x_0\neq x_1$ has dp-rank 2 but does not contain a rectangle.

%

\medskip
\textbf{Acknowledgements} Thanks to Sergei Starchenko for motivating me to work on Conjecture \ref{conj_main}. Thanks also to Itay Kaplan for reading some previous versions of this paper. Finally, many thanks to the referee for a careful reading of the paper and a number of useful comments and corrections.

\section{Setting and basic facts}

Throughout, $T$ is a complete theory, which we do not always assume to be NIP and $\monster$ is a monster model. We write $M \pprec N$ to mean $M\prec N$ and $N$ is $|M|^+$-saturated.

The notation $\phi^0$ means $\neg \phi$ and $\phi^1$ means $\phi$.

If $M\pprec N$ and $p\in S(N)$, then $p$ is $M$-invariant if for any $b,b'\in N$ and any formula $\phi(x;y)$, $b\equiv_M b'$ implies $p\vdash \phi(x;b) \leftrightarrow \phi(x;b')$. If $N$ is not specified, we mean $N=\monster$. If $M$ is omitted, we mean ``for some $M$ such that $M\pprec N$". Any $M$-invariant type over $N$ extends in a unique way to a global $M$-invariant type. Thus there is no harm in considering only global invariant types.

Let $I=( a_i:i\in \ordI)$ be any sequence. We define the \emph{Ehrenfeucht-Mostowski type} (or \emph{EM-type}) of $I$ over $A$ to be the set of $L(A)$-formulas $\phi( x_1,\ldots, x_n)$ such that $\monster \models \phi( a_{i_1},\ldots, a_{i_n})$ for all $i_1 < \cdots < i_n\in \ordI$, $n<\omega$. If $I$ is an indiscernible sequence, then for every $n$, the restriction of the EM-type of $I$ to formulas in $n$ variables is a complete type over $A$. If $I$ is any sequence and $\ordJ$ is any infinite linear order, then using Ramsey's\index{Ramsey} theorem and compactness, we can find an indiscernible sequence $J$ indexed by $\ordJ$ and realising the EM-type of $I$ (see \cite[Lemma 5.1.3]{TentZieg}).

\smallskip
As usual, we say that two types $p,q\in S(N)$ are \emph{weakly orthogonal} if $p(x)\cup q(y)$ defines a complete type in two variables over $N$. If $p$ and $q$ are $M$-invariant types, we say they are \emph{orthogonal} if they are weakly orthogonal as global types. Note that this implies that $p|_N$ and $q|_N$ are weakly orthogonal for any $N$ such that $M\pprec N$.

An important notion in this work is that of commuting types. If $p(x)$ and $q(y)$ are two global invariant types, then $p(x)\otimes q(y)$ denotes $\tp(a,b/\monster)$ where $b\models q$ and $a\models p|\monster b$. We say that $p$ and $q$ \emph{commute} if $p(x)\otimes q(y) = q(y) \otimes p(x)$. We say that $p$ and $q$ commute over $M_1$ if $p(x)\otimes q(y)|_{M_1} = q(y) \otimes p(x)|_{M_1}$. Note that by associativity of $\otimes$, if $p$ and $q$ commute, then $p$ commutes with $q^{(n)}=q\otimes \cdots \otimes q$.

The following observation will be used frequently: Assume that $p$ and $q$ are $M$-invariant global types. Let $M\pprec N$. Build successively $b\models q|_N$, $a\models p|_{Nb}$ and $I$ a Morley sequence of $q$ over $Nab$. Then the sequence $b+I$ is indiscernible over $Na$ if and only if $p$ and $q$ commute.

Of course, if $p$ and $q$ are orthogonal, then they commute. In NIP theories, we can consider commuting as a kind of weak form of orthogonality. This may seem exaggerated since for example a type may commute with itself, but it turns out to be a useful intuition. It is also motivated by the study of distal theories (see \cite{distal}) where in fact the two notions coincide (and this can be taken as a definition of distal theories amongst NIP theories).

\smallskip
Recall that, in an NIP theory, a global invariant type $p$ is \emph{generically stable} if it is both definable and finitely satisifiable over a small model. This is equivalent to saying that $p$ commutes with itself: $p(x_1)\otimes p(x_2)=p(x_2)\otimes p(x_1)$ (see \cite[Theorem 2.29]{NIPbook}).

\smallskip
We recall also the notion of strict non-forking from \cite{CherKapl}. Let $M$ be a model of an NIP theory. A sequence $(b_i)_{i<\omega}$ is strictly non-forking over $M$ if for each $i<\omega$, $\tp(b_i/b_{<i}M)$ is strictly non-forking over $M$ which means that it extends to a global type $\tp(b_*/\monster)$ such that both $\tp(b_*/\monster)$ and $\tp(\monster/Mb_*)$ are non-forking over $M$. We will only need to know two facts about strict non-forking sequences:

\smallskip
(Existence) Given $b\in \monster$ and $M\models T$, there is a sequence $b=b_0,b_1,\ldots$ which is strictly non-forking over $M$. We might call such a sequence a \emph{strict Morley sequence} of $\tp(b/M)$.

\smallskip
(Witnessing property) If the formula $\phi(x;b)$ forks over $M$, then for any strictly non-forking sequence $b=b_0,b_1,\ldots$, the type $\{\phi(x;b_i):i<\omega\}$ is inconsistent.

\smallskip
If $\phi(x;y)$ is an NIP formula, we let $\alt(\phi)$ be the \emph{alternation number} of $\phi$, namely the maximal $n$ for which there is an indiscernible sequence $(b_i:i<\omega)$ and a tuple $a$ with $\neg (\phi(a;b_i)\leftrightarrow \phi(a;b_{i+1}))$ for all $i<n$. If $(b_i:i<\omega)$ is indiscernible and $\{\phi(x;b_i):i<\alt(\phi)/2+1\}$ is consistent, then $\{\phi(x;b_i):i<\omega\}$ is also consistent.

\medskip

\subsection{Dp-rank and dp-minimality}\label{sec_dprankintro}

Sequences $(I_i:i<k)$ are said to be \emph{mutually indiscernible} over $A$ if each $I_i$ is indiscernible over $A\cup I_{\neq i}$.

We briefly recall the definition of dp-rank and refer the reader to \cite[Chapter 4]{NIPbook} for more information. We will only need to consider finite dp-ranks, hence we restrict our definition to this case.

\begin{defi}
Let $\pi(x)$ be a partial type over some set $A$ and $n<\omega$. Then $\pi(x)$ is of dp-rank $\leq n$ if for any $a\models \pi(x)$ and any $n+1$ sequences $I_0,\ldots,I_{n}$ mutually indiscernible over $A$, there is $k\leq n$ such that $I_k$ is indiscernible over $Aa$.

The partial type $\pi$ is of dp-rank $n$ (written $\dpr(\pi(x))=n$) if it is of dp-rank $\leq n$, but not $\leq n-1$.
\end{defi}

One can check that this definition does not depend on the choice of $A$ over which $\pi(x)$ is defined.

We will often write $\dpr(a/A)$ instead of $\dpr(\tp(a/A))$.

If $\pi(x)$ is a partial type over $A$, then the definition implies immediately that $\dpr(\pi(x))=\max_p \dpr(p)$ where $p$ ranges over all complete $A$-types extending $\pi(x)$.

\begin{prop}[\cite{KOU}]\label{prop_addit}
Dp-rank is sub-additive: for any $a, b$ and $A$, we have $\dpr(a, b/A)\leq \dpr( a/A b)+\dpr( b/A)$.
\end{prop}

Equality need not hold as we will see in Section \ref{sec_dprank}. However, it is always the case that $\dpr(\phi( x)\wedge \psi( y))=\dpr(\phi( x))+\dpr(\psi( y))$, where $ x$ and $ y$ are disjoint tuples of variables.

A theory $T$ is \emph{dp-minimal} if every one-type has dp-rank at most 1. Equivalently, $T$ is dp-minimal if for every set $A$, infinite sequences $I_0,I_1$ of tuples and every element $a$, if $I_0$ is indiscernible over $AI_1$ and $I_1$ is indiscernible over $AI_0$, then either $I_0$ or $I_1$ is indiscernible over $Aa$. By Proposition \ref{prop_addit}, every $n$-tuple $\bar a$ has dp-rank at most $n$ (over any set $A$).

Any dp-minimal theory is NIP.

\section{Invariant types}\label{sec_invtypes}
Our guiding conjecture in this section is the following which first appeared in \cite{ExtDef2}.

\begin{conjecture}\label{conj_main}
Let $T$ be NIP and $M\models T$. Let $\phi(x;d)\in L(\monster)$ be a formula, non-forking over $M$. Then there is $\theta(y)\in \tp(d/M)$ such that the partial type $\{ \phi(x;d') : d'\in \theta(\monster)\}$ is consistent.
\end{conjecture}

\noindent
First, a few basic observations:

$\cdot$ As $\phi(x;d)$ does not fork over $M$, it extends to some $M$-invariant type $p$. (Recall that in NIP theories, non-forking and invariance are the same over models, \cite{CherKapl}.)

$\cdot$ If $p$ is finitely satisfiable, then in particular, $\phi(x;d)$ has a solution $a$ in $M$. Then we can take $\theta(y)=\phi(a;y)$. In this case, the formula $\phi(x;d)$ also extends to the definable type $x=a$.

$\cdot$ If $p$ is definable, then we may take $\theta(y)$ to be the $\phi$-definition of $p$.

\smallskip \noindent
Hence the interesting case is when $p$ is neither definable nor finitely satisfiable. This is where the ideas mentioned in the introduction become useful.

\subsubsection*{The $(p,q)$-theorem}

We note that this conjecture can be seen as a definable version of the $(p,q)$-theorem from finite combinatorics; the statement of which we recall now. Let $\phi(x;y)$ be a formula. We define the dual VC-dimension of $\phi(x;y)$ as the maximal $n<\omega$ (if it exists) for which there are tuples $b_0,\ldots,b_{n-1}$ and $(a_C : C\subseteq n)$ such that $$\models \phi(a_C;b_k) \iff k\in C.$$

If such a maximal $n$ does not exist, we say that $\phi(x;y)$ has infinite dual VC-dimension. A formula has finite dual VC-dimension if and only if it is NIP (see \cite[Chapter 6]{NIPbook}).

\begin{fait}[$(p,q)$-theorem]\label{fact_pq}
Given integers $p\geq q$, there is an integer $n$ such that the following holds.
Let $\phi(x;y)$ be a formula of dual VC-dimension $<q$ and let $W\subseteq M^{|y|}$ be the set of tuples $b$ for which $\phi(x;b)$ is not empty. Let $Y\subset W$ be finite and assume that
for every $Y_0\subseteq Y$ of size $p$, we can find $Y_1\subseteq Y_0$ of size $q$ and $a\in M$ such that $Y_1\subseteq \phi(a;M)$,
 \underline{then} there are $a_0,\dots,a_{n-1}$ such that $Y\subseteq \bigvee_{i<n} \phi(a_i;M)$.
\end{fait}

See \cite[Chapter 6]{NIPbook} for more details and for a proof of a special case. This theorem was used in \cite{ExtDef2} to prove uniformity of honest definitions. We refer the reader to \cite{pq} for the original proof.

Let $T$ be NIP and $M\models T$. Assume that $\phi(x;d)\in L(\monster)$ does not fork over $M$. By lowness (see \cite[Proposition 5.38]{NIPbook}), there is $\psi(y)$ in $L(M)$ such that for any $d'\models \psi(y)$, the formula $\phi(x;d')$ does not fork over $M$. Let $W\subseteq S_y(M)$ be the set of types containing the formula $\psi(y)$. As noted in \cite[Proposition 25]{ExtDef2}, the $(p,q)$-theorem implies that we can write $W=\bigcup_{i<n} W_i$ such that for each $i<n$, $\{\phi(x;d'): d'\in \monster, \tp(d'/M)\in W_i\}$ is consistent (and thus does not fork over $M$).

Conjecture \ref{conj_main} and compactness imply that we can choose the sets $W_k$ to be clopens. In fact, the converse is also true: if we can choose the $W_k$'s to be definable, then Conjecture \ref{conj_main} follows since $\tp(b/M)$ must lie in one of them.

\smallskip
Finally, note that it is enough to prove the conjecture when $T$ is countable, because we can restrict to a countable $T$ containing $\phi(x;y)$. Then we can also assume that $M$ is countable: if it is not, we can replace it with a countable submodel over which $\phi(x;d)$ does not fork.

\subsection{Recognising definable types}

%
%

The following lemma holds in any theory and is the key to identifying definable types.

\begin{lemme}\label{lem_def}
An $M$-invariant type $p(x)$ is definable if and only if for every $M$-finitely satisfiable type $q(y)$, $p(x)\otimes q(y)|_M=q(y)\otimes p(x)|_M$.
\end{lemme}
\begin{proof}
If $p$ is definable, then it is known (and easy to see) that it commutes with every finitely satisfiable type (see \cite[Lemma 2.23]{NIPbook}). Conversely, assume that $p$ commutes with every $M$-finitely satisfiable type as in the statement of the proposition. We first show that $p$ is an heir of its restriction to $M$. Assume that this is not the case. Then there is $\phi(x;y)\in L(M)$ and $d\in \monster$ such that $p\vdash \phi(x;d)$ and for all $b\in M$, $p\vdash \neg \phi(x;b)$. Let $q$ be any global coheir of $\tp(d/M)$. Then $p(x)\otimes q(y)\vdash \phi(x;y)$ by construction, but necessarily, $q(y)\otimes p(x)\vdash \neg \phi(x;y)$. This contradicts the hypothesis.

To conclude it is now enough to show that $p|_M$ has a unique heir to $\monster$. Let $p_0,p_1$ be two global heirs of $p|_M$. If $p_0\neq p_1$, then for some formula $\phi(x;b)\in L(\monster)$, we have $p_0\vdash \phi(x;b)$ and $p_1\vdash \neg \phi(x;b)$. Let $a_0\models p_0|_{Mb}$ and $a_1\models p_1|_{Mb}$. By the heir property, we know that both $\tp(b/Ma_0)$ and $\tp(b/Ma_1)$ are finitely satisfiable in $M$. Let $q_0$ (resp. $q_1$) be a global extension of $\tp(b/Ma_0)$ (resp. $\tp(b/Ma_1)$) which is finitely satisfiable in $M$. As both $a_0$ and $a_1$ realise $p|_M$, we have $q_0(y)\otimes p(x)\vdash \phi(x;y)$ whereas $q_1(y)\otimes p(x)\vdash \neg \phi(x;y)$. But as $q_0|_M=q_1|_M$, we have $p(x)\otimes q_0(y)\vdash \phi(x;y) \iff p(x)\otimes q_1(y)\vdash \phi(x;y)$. We get a contradiction to the commutativity hypothesis.
\end{proof}

In the study of the dp-minimal case, we will work by induction on the number of variables. Hence the following will be useful.

\begin{lemme}\label{lem_indfs}
Let $T$ be NIP. Let $\phi(x,y;d)\in L(\monster)$ and $M\models T$ such that $\phi(x,y;d)$ does not fork over $M$. Assume that there are $(a,b)\models \phi(x,y;d)$, $\tp(a,b/\monster)$ is $M$-invariant and $\tp(b/\monster)$ is finitely satisfiable in $M$. Then there is $b_0 \in M$ such that $\phi(x;b_0,d)$ does not fork over $M$.
\end{lemme}
\begin{proof}
Let $(d_i :i<\omega)$ be a strict Morley sequence of $\tp(d/M)$. Let $n$ be larger than the alternation number of $\phi(x;y)$. Then $b$ satisfies the formula $(\exists x)\bigwedge_{k<n} \phi(x,y;d_k)$. Therefore there is $b_0\in M$ satisfying the same formula. We claim that $\phi(x;b_0,d)$ does not fork over $M$. Indeed, there is $a'$ such that $\bigwedge_{k<n} \phi(x;b_0,d_k)$ holds. By hypothesis on $n$, this implies that the type $\{\phi(x;b_0,d_i) : i<\omega\}$ is consistent and therefore $\phi(x;b_0,d)$ does not fork over $M$.
\end{proof}

\begin{cor}
Assume that $T$ is NIP. If all invariant 1-types are finitely satisfiable in a small model, then all invariant types are.
\end{cor}

Now to show Conjecture \ref{conj_main} by induction it would be enough to consider the case where for every $M$-invariant type extending $\phi(\bar x;d)$ none of the induced one-types are finitely satisfiable.

\subsection{One variable}

\begin{lemme}\label{lem_techprem}
Let $B\subset \monster$ and let $a\in \monster$ be a tuple such that $\dprk(a/B)=n$. Let $\bar b_1,\ldots, \bar b_n$ in $\monster$ be infinite sequences, mutually indiscernible over $B$, none of which is indiscernible over $Ba$. Let $\phi(x;y)\in L$, $|x|=|a|$.

Then there are formulas $\psi(x)\in \tp(a/B\bar b_1..\bar b_n)$ and $\theta_l(y)\in L(B\bar b_1..\bar b_n)$ $l=0,1$, such that:

$\bullet_0$ for each $b\in B^{|y|}$, one of $\theta_0(b)$ or $\theta_1(b)$ holds;

$\bullet_1$ for $l=0,1$, $\monster \models \theta(y)\wedge \psi_l(x) \rightarrow \phi^{l}(x;y)$.
\end{lemme}
\begin{proof}
Let $r\in S_y(B\bar b_1..\bar b_n)$ be finitely satisfiable in $B$. Let $r'$ be any global extension of $r$ to a type finitely satisfiable in $B$. Let $l=l(r') \in \{0,1\}$ be such that $r'\models \phi^l(a;y)$.

Assume that we can find $c\models r$ such that $\models \neg\phi^{l}(a;c)$. Then let $J$ be a Morley sequence of $r'$ over everything and $\bar c=c+J$. The sequences $\bar b_1,\ldots,\bar b_n,\bar c$ are mutually indiscernible over $B$ and none of them is indiscernible over $Ba$. This contradicts the fact that $\dpr(a/B)=n$. Thus by compactness, there are $\theta_r(y)\in r$ and $\psi_r(x)\in \tp(a/B\bar b_1..\bar b_n)$ such that $\models \theta_r(y)\wedge \psi_r(x) \impl \phi^{l}(x;y)$. In particular, $l$ depends only on $r$, not on $r'$, and we can write $l=l(r)$.

Let $S\subset S_y(B\bar b_1..\bar b_n)$ be the set of types finitely satisfiable in $B$. It is a closed set, thus compact and contains all types realised in $B$. We can extract from the family $\{\theta_r(y) : r\in S\}$ a finite subcover $\{\theta_r(y) : r\in S^*\}$. For $l=0,1$, let $S^*_l = \{r \in S^* : l(r)=l\}$ and define $\theta_l(y)=\bigvee_{r\in S^*_l} \theta_r(y)$. Also define $\psi(x)=\bigwedge_{r\in S^*} \psi_r(x)$.

We have that $\theta_0,\theta_1$ cover $S$, in particular, $\theta_0(B)\cup \theta_1(B)=B^{|y|}$. Also for $l=0,1$, $\psi(x)\in \tp(a/B\bar b_1..\bar b_n)$ and $\monster \models \theta_l(y)\wedge \psi(x)\impl \phi^{l}(x;y)$.
\end{proof}

\begin{lemme}\label{lem_tech}
Let $M\prec^+N_1 \prec^+ N$ and let $a\in \monster$ be a tuple such that $\dprk(a/N)=n$. Let $\bar b_1,\ldots, \bar b_n$ in $\monster$ be infinite sequences, mutually indiscernible over $N$, none of which is indiscernible over $Na$. Assume also that $\tp(a\bar b_1..\bar b_n/N)$ is $M$-invariant. Then \[\tp(a/N_1\bar b_1..\bar b_n)\vdash \tp(a/N).\]

More precisely, given $\phi(x;y)\in L$, $|x|=|a|$, there are formulas $\theta_l(y)\in L(N_1\bar b_1..\bar b_n)$ ($l=0,1$) and $\psi(x)\in \tp(a/N_1\bar b_1..\bar b_n)$ such that:

$\bullet_0$ for each $b\in N^{|y|}$, one of $\theta_0(b)$ or $\theta_1(b)$ holds;

$\bullet_1$ for $l=0,1$, $\monster \models \theta_l(y)\wedge \psi(x) \rightarrow \phi^{l}(x;y)$.

\end{lemme}
\begin{proof}
Let $\psi(x)$, $\theta_l(y)$ be given by Lemma \ref{lem_techprem} with $B=N$.

Write $\theta_l(y)=\theta_l(y;\bar b_1,\ldots,\bar b_n, e)$ and $\psi(x)=\psi(x;\bar b_1,\ldots,\bar b_n, e)$ with $e\in N$. As $\tp(a\bar b_1..\bar b_n/N)$ is $M$-invariant, we may replace $e$ by any $e' \equiv_{M} e$. In particular, we may assume that $e\in N_1$. This gives what we want.
\end{proof}

Our first theorem is stated for a type of dp-rank 1 in an arbitrary theory.

\begin{thm}\label{th_invtype2}
($T$ any theory) Let $p$ be a global $M$-invariant type of dp-rank 1. Then $p$ is either finitely satisfiable or definable.
\end{thm}
\begin{proof}
Assume that $p$ is not definable. Then there is a global type $q$ finitely satisfiable in $M$ such that $p$ does not commute with $q$. Take $N\succ M$ sufficiently saturated. Let $\phi(x;y)\in L$, $d\in N$ such that $\phi(x;d)\in p$.

Let $(a,b)\models p\otimes q|_N$, then let $I$ be a Morley sequence of $q$ over $Nab$ and let $\bar b=b+I$. The sequence $\bar b$ is indiscernible over $N$, but not over $Na$. Let $M\pprec N_1 \pprec N$ with $\tp(N_1/Md)$ finitely satisfiable in $M$.

Apply Lemma \ref{lem_tech} to $a$, $N_1,N$ and $\bar b$, with $n=1$. The second part of the conclusion gives formulas $\theta_l(y),\psi(y) \in L(N_1\bar b)$. Write $\theta_l(y)=\theta_l(y;\bar b, e)$ and $\psi(x)=\psi(x;\bar b, e)$ with $e\in N_1$.

%
%
%
%
%
%

Since $\phi(x;d)\in p$, we know that the formula $\theta_1(d;\bar b, e)$ holds. As $\tp(\bar b/N)$ is finitely satisfiable in $M$, there is $\bar b_0\in M$ such that 
$$\bar b_0 \models \theta_1(d;\bar z, e)\wedge (\exists x)(\forall y)(\theta_1(y;\bar z, e)\impl \phi(x;y)).$$
Since $N_1$ is a model, there is $a_0 \in N_1$ such that $(\forall y)(\theta_1(y;\bar b_0,e)\rightarrow \phi(a_0;y))$ holds. In particular $\phi(a_0;d)$ holds. As $\tp(N_1/Md)$ is finitely satisfiable in $M$, we can find $a_0'\in M$ satisfying $\phi(x;d)$. As $\phi(x;d)$ was any formula in $p$, this proves that $p$ is finitely satisfiable in $M$.
\end{proof}

Hence if $T$ is dp-minimal, Conjecture \ref{conj_main} is proved for any formula $\phi(x;d)$, $|x|=1$.

\subsection{Two variables}

In this section and the next one, we assume for simplicity that $T$ is dp-minimal and try to deal with formulas in more than one variable.

The next proposition solves Conjecture \ref{conj_main} for formulas in two variables.

\begin{prop}\label{prop_dim2}
Assume that $T$ is dp-minimal. Let $\phi(x_1,x_2;d)\in L(\monster)$ be non-forking over $M$; $|x_1|=|x_2|=1$. Then there is $\theta(y)\in \tp(d/M)$ such that $\bigwedge_{d'\in \theta(\monster)} \phi(x_1,x_2;d')$ is consistent.
\end{prop}
\begin{proof}
Fix some model $M\prec N$, $N$ is very saturated and let $\phi(x_1,x_2;d)\in L(N)$ be non-forking over $M$. Let $a_1\hat{~}a_2\models \phi(x_1,x_2;d)$ such that $\tp(a_1,a_2/N)$ is non-forking over $M$. By Lemma \ref{lem_indfs}, we may assume that $p_1=\tp(a_1/N)$ is not finitely satisfiable in $M$. Therefore it is definable and since it is not generically stable, it does not commute with itself. Also we may assume that $p=\tp(a_1,a_2/N)$ is not definable, therefore there is some type $q\in S(N)$ finitely satisfiable in $M$ such that $p$ does not commute with $q$.

Now let $c_1,c_2\in \monster$ such that $(a_1\hat{~}a_2,c_2,c_1)\models p\otimes q\otimes p_1|_N$. Let $I$ be a Morley sequence of $p_1$ over everything and $J$ a Morley sequence of $q$ over everything. Then the sequences $\bar c_1=c_1+I$ and $\bar c_2=c_2+J$ are mutually indiscernible over $N$ (because the types $p_1$ and $q$ commute). But neither of them is indiscernible over $Na_1a_2$. Take some $M\pprec N_1\pprec N$ such that $\tp(N_1/Md)$ is finitely satisfiable in $M$. As the dp-rank of $a_1\hat{~}a_2$ over $N$ is 2 (by Proposition \ref{prop_addit}), we can apply Lemma \ref{lem_tech}. We conclude that $\tp(a_1\hat{~}a_2/\bar c_1\bar c_2N_1)\vdash \tp_\phi(a_1\hat{~}a_2/N)$ as witnessed by some $\psi(x;\bar c_1,\bar c_2, e)$, $\theta_l(y;\bar c_1,\bar c_2, e)$, $l=0,1$, with $e\in N_1$.

As $\tp(\bar c_2/N\bar c_1)$ is finitely satisfiable in $M$, there is $\bar c_2'\in M$ such that:
$$\models \theta_1(d;\bar c_1,\bar c'_2, e) \wedge (\exists x)(\forall y)(\theta_1(y;\bar c_1,\bar c'_2, e)\impl \phi(x;y)).$$

Let $\Theta(d;\bar c_1,\bar c'_2, e)$ be this conjunction. Since $\tp(\bar c_1/N)$ is definable over $M$, there is a formula $d\Theta(y;\bar z_2,\bar t)\in L(M)$ such that for all $y,\bar z_2,\bar t\in N$, we have $d\Theta(y;\bar z_2,\bar t)\leftrightarrow \Theta(y;\bar c_1,\bar z_2,\bar t)$. As $\tp(e/Md)$ is finitely satisfiable in $M$, we can find $e'\in M$ such that $d\Theta(d;\bar c'_2, e')$ holds. Then unwinding, we see that the type $\{\phi(x;d') : d' \models d\Theta(y;\bar c'_2, e')\}$ is consistent, as required.
\end{proof}


\subsection{More variables}

The proof of the two-variable case relied on the fact that non-forking formulas in one variable extend to definable types. However the conclusion we obtain is weaker and this prevents us from going on to higher arities. In this section, we do our best to pursue nonetheless. We manage to make an induction go through, but with an even weaker property.

In this section, we assume that $T$ is countable.

\smallskip
We start with a local version of Lemma \ref{lem_def}.

\begin{lemme}\label{lem_comloc}
Let $M\prec^{+} N$ and let $a\in \monster$ such that $p=\tp(a/N)$ is $M$-invariant. Let $q\in S_y(M)$ and $b\in q(N)$. The following are equivalent:

(i) $p\otimes \tilde q|_M = \tilde q\otimes p|_M$ for every global coheir $\tilde q$ of $q$.

(ii) for every formula $\phi(x;y)\in L(M)$ such that $a\models \phi(x;b)$, there is $\theta(y)\in q$ such that for any $b'\in \theta(M)$, $a\models \phi(x;b')$.
\end{lemme}
\begin{proof}
(i) $\Rightarrow$ (ii): Assume that (i) holds and let $\phi(x;y)\in L(M)$ such that $a\models \phi(x;b)$. Then (i) implies that there is no coheir $\tilde q$ of $q$ such that $\tilde q\models \neg \phi(a;y)$. In other words, $q\cup \{\neg \phi(a;y)\}$ is not finitely satisfiable in $M$. This exactly means that for some $\theta(y)\in q$, $\theta(M)\cap \neg\phi(a;M)=\emptyset$, hence $\phi(a;b')$ holds for every $b'\in \theta(M)$.

(ii) $\Rightarrow$ (i): Assume (ii). Let $\phi(x;y)\in L(M)$ such that $p\otimes q\vdash \phi(x;y)$ and $\theta(y)\in L(M)$ given by (ii). Let $\tilde q$ a global coheir of $q$, and we have to show that $\tilde q \vdash \phi(a;y)$. Assume not, then $\tilde q\vdash \neg \phi(a;y)\wedge \theta(y)$. But that formula is not realised in $M$. Contradiction.
\end{proof}

Hence to solve Conjecture \ref{conj_main}, it is enough to prove that given $M\prec^{+} N$ and $\phi(x;d)\in L(N)$ non-forking over $M$, there is $a\in \phi(\monster;d)$, $\tp(a/N)$ does not fork over $M$ and commutes over $M$ with every coheir of $\tp(d/M)$. (Because then, taking $\theta(y)\in \tp(d/M)$ as in point (ii) of the lemma, we have that $\{\phi(x;b'):b'\in \theta(\monster)\}$ is consistent.)

We will not succeed in finding such a type, instead we will obtain a weaker property.
\begin{lemme}\label{lem_comloc2}
Let $M\prec M_1\prec^{+} N$ and let $a\in \monster$ such that $p=\tp(a/N)$ is $M$-invariant. Let $q\in S_y(M)$ and $b\in q(N)$. The following are equivalent:

(i) $p\otimes \tilde q|_{M_1} = \tilde q\otimes p|_{M_1}$ for every global coheir $\tilde q$ of $q$ extending $\tp(b/M_1)$;

(ii) for every formula $\phi(x;y)\in L(M_1)$ such that $a\models \phi(x;b)$, there is $\theta(y)\in \tp(b/M_1)$ such that for any $b'\in \theta(M)$, $a\models \phi(x;b')$.
\end{lemme}
\begin{proof}
We assume that $\tp(b/M_1)$ is finitely satisfiable in $M$, otherwise everything is trivial.

(i) $\Rightarrow$ (ii): Assume that (i) holds and let $\phi(x;y)\in L(M_1)$ such that $a\models \phi(x;b)$. Then (i) implies that there is no global coheir $\tilde q$ of $q$ exten\-ding $\tp(b/M_1)$ such that $\tilde q\models \neg \phi(a;y)$. Therefore the partial type $\tp_y(b/M_1) \cup \{\neg \phi(a;y)\}$ is not finitely satisfiable in $M$. Hence there is $\theta(y)\in \tp(b/M_1)$ such that $\theta(M)\subseteq \phi(a;M)$.

(ii) $\Rightarrow$ (i): Assume (ii). Let $\phi(x;y)\in L(M_1)$ such that $p\otimes q\vdash \phi(x;y)$ and $\theta(y)\in L(M_1)$ given by (ii). Let $\tilde q$ a global coheir of $q$ extending $\tp(b/M_1)$, and we have to show that $\tilde q \vdash \phi(a;y)$. Assume not, then $\tilde q\vdash \neg \phi(a;y)\wedge \theta(y)$. But that formula is not realised in $M$. Contradiction.
\end{proof}

We introduce the notion of ``$a_2$-forking" as defined in Cotter \& Starchenko's paper \cite{CotStar}. For this, we assume that $T$ is NIP.

Assume we have $M\prec^{+} N$ and $a_2\in \monster$ such that $\tp(a_2/N)$ is $M$-invariant. We say that a formula $\psi(x,a_2;d)\in L(Na_2)$ $a_2$-divides over $M$ if there is an $M$-indiscernible sequence $(d_i:i<\omega)$ inside $N$ with $d_0=d$ and $\{\psi(x,a_2;d_i):i<\omega\}$ is inconsistent. We define $a_2$-forking in the natural way: the formula $\psi(x,a_2;d)$ $a_2$-forks over $M$ if it implies a finite disjunction of formulas $\psi_i(x,a_2;d_i)\in L(Na_2)$ each of which $a_2$-divides over $M$.

\begin{fait}\label{fact_betafork}
Notations being as above, the following are equivalent:

(i) $\psi(x,a_2;d)$ does not $a_2$-divide over $M$;

(ii) $\psi(x,a_2;d)$ does not $a_2$-fork over $M$;

(iii) if $(d_i:i<\omega)$ is a strict Morley sequence of $\tp(d/M)$ inside $N$, then $\{\psi(x,a_2;d_i):i<\omega\}$ is consistent;

(iv) there is $a_1\models \psi(x,a_2;d)$ such that $\tp(a_1,a_2/N)$ is $M$-invariant.
\end{fait}

The proof of the equivalences of (i)-(iii) can be found in the Appendix of \cite{CotStar}. The proof is an easy adaptation of the corresponding facts for usual dividing and forking proved in \cite{CherKapl}. It is assumed in \cite{CotStar} that $\tp(a_2/N)$ is $M$-definable, but this is only used through Remark 5.11 there which only needs $M$-invariance.

It remains to show the equivalence to (iv). It is clear that (iv) implies (iii). Conversely assume that (iv) does not hold. Then $\psi(x,a_2;d)$ implies a finite disjunction of formulas of the form $\theta(x,a_2;e,e')=\neg (\zeta(x,a_2;e)\leftrightarrow \zeta(x,a_2;e'))$, with $e,e'\in N$, $\tp(e/M)=\tp(e'/M)$. By NIP, the formula $\theta(x,y;e,e')$ divides over $M$ (since it does not extend to an invariant type) hence $\theta(x,a_2;e,e')$ $a_2$-divides over $M$ which implies that $\psi(x,a_2;d)$ $a_2$-forks over $M$.

\medskip
Let $M\prec^{+} N$, $M$ is countable and $d\in N$. We can find $M\prec M_1 \prec N$ such that:

$\bullet_0$ $M_1$ is countable;

$\bullet_1$ $\tp(d/M_1)$ is finitely satisfiable in $M$;

$\bullet_2$ for every finite $m\in M_1$, there is $d'\in M_1$ such that $(m,d')\equiv_{M} (m,d)$;

$\bullet_3$ for every finite $m\in M_1$, there is a strict Morley sequence $(m_i:i<\omega)$ of $\tp(m/M)$ in $M_1$, with $m_0=m$;

\smallskip
It easy to build such a model in $\omega$ steps: first fix a global coheir $\tilde q$ of $\tp(d/M)$. Let $M^0=M$, and having built $M^k$ take $M^{k+1}\supseteq M^k$ to satisfy $\bullet_2$, $\bullet_3$ where $m$ is taken in $M^k$.
Then move $M^{k+1}$ so that $\tp(d/M^{k+1})=\tilde q|M^{k+1}$. Having done this inductively for all $k$, let $M_1$ be the union of the $M^k$'s.

\begin{prop}\label{prop_mainprop}
Assume $T$ is countable and dp-minimal. Let $M\prec M_1\prec^{+} N$, and $d\in N$ as above. Let $a_1,a_2\in \monster$, $|a_1|=1$, such that $p=\tp(a_1,a_2/N)$ is $M$-invariant, $\phi_0(a_1,a_2;d)$ holds for some $d\in N$ and $\tp(a_2/N)$ commutes over $M_1$ with every coheir of $\tp(d/M)$ extending $\tp(d/M_1)$.

Then there is $a'_1\in \monster$ such that $\phi_0(a'_1,a_2;d)$ holds, $\tp(a'_1,a_2/N)$ is $M$-invariant  and commutes over $M_1$ with every coheir of $\tp(d/M)$ extending $\tp(d/M_1)$.
\end{prop}
\begin{proof}
We say that a tuple $a\hat{~}a_2$, $|a|=|a_1|$, is $\Gamma$-good, when:

-- $\Gamma = \{ (\psi_\phi (x_1,x_2),\theta_\phi(x_2;d) ) : \phi\in \Gamma_L\}$; where $\psi_\phi, \theta_\phi \in L(M_1)$;

-- $\Gamma_L \subseteq L(M_1)$ and each $\phi\in \Gamma_L$ has the form $\phi(x_1,x_2;y)$, $|x_i|=|a_i|$, $|y|=|d|$;

-- for each $\phi(x_1,x_2;y)\in \Gamma_L$, we have

 $\models \psi_\phi(a,a_2) \wedge \theta_\phi(a_2;d)$ and
 
  $\models (\forall x,y)( \psi_\phi(x,a_2)\wedge \theta_\phi(a_2;y) \impl \phi(x,a_2;y))$.

\smallskip \noindent
\underline{Claim 1}:
If $a\hat{~}a_2$ is $\Gamma$-good, where

 $\Gamma_L=\{\phi\in L(M_1): \phi=\phi(x_1,x_2;y)\in\tp(a,a_2,d/M_1) \}$,
 
  then $\tp(a,a_2/N)$ commutes over $M_1$ with every coheir of $\tp(d/M)$ extending $\tp(d/M_1)$.

Proof: We have to check condition (ii) of the Lemma \ref{lem_comloc2}. Let $\phi(x_1,x_2;y)\in L(M_1)$ such that $a\hat{~}a_2\models \phi(x_1,x_2;d)$. By hypothesis on $\tp(a_2/N)$, and Lemma \ref{lem_comloc2}, there is $d\theta(y)\in \tp(d/M_1)$ such that for every $b\in M$, we have $d\theta(b)\impl \theta_\phi(a_2;b)$. Then for each $b\in d\theta(M)$, we have $\models \phi(a,a_2;b)$. Hence the claim is proved.

\medskip
Fix an enumeration of formulas $\phi(x_1,x_2;y)$ in $L(M_1)$, of order type $\omega$ for which $\phi_0$ is the first formula. Assume that we are given $a_1,a_2$ as in the statement and some $\Gamma$, $\Gamma_L$ finite, such that $a_1\hat{~}a_2$ is $\Gamma$-good.  Let $\psi_*(x_1,x_2;g)\in L(M_1)$ be the conjunction of $\psi_\phi(x_1,x_2)$ for $\phi\in \Gamma_L$. If $a_1\hat{~}a_2$ commutes over $M_1$ with every coheir of $\tp(d/M)$ extending $\tp(d/M_1)$, we are done. Otherwise, there is such a coheir $\tilde q$ such that $p=\tp(a_1,a_2/N)$ does not commute with $\tilde q$ over $M_1$. Let $I\hat{~}a_1a_2 \hat{~}b$ realise $\tilde q^{\omega} \otimes p\otimes \tilde q$ over $N$ and set $\bar b=b+I$. Then $\bar b$ is indiscernible over $M_1 a_2$ (because $\tp(a_2/N)$ commutes with $\tilde q$ over $M_1$), but it is not indiscernible over $M_1 a_1a_2$ by assumption.

Let $\phi(x_1,x_2;y)$ be the least formula in $L(M_1)$ which is not in $\Gamma_L$ and such that $\phi(a_1,a_2;d)$ holds. We can apply Lemma \ref{lem_techprem} with $(a,B,n,\bar b_1..\bar b_n)$ there being $(a_1,M_1a_2,1,\bar b)$ here. It gives us formulas $\psi(x_1,a_2;\bar b,e)\in \tp(a_1/M_1a_2\bar b)$, $\theta_l(y)=\theta_l(a_2,\bar b,e;y)$, $l=0,1$, $e\in M_1$ such that:

$(*)$\quad $\psi(x_1,a_2;\bar b,e)\wedge \theta_l(a_2,\bar b,e;y)\rightarrow \phi^l(x_1,a_2;y)$.
%

\smallskip \noindent
\underline{Claim 2}: We have $\models \theta_1(d)$.

Proof: We first show that $\tp(d/M_1 a_2 \bar b)$ is finitely satisfiable in $M_1$. So let $\zeta(y;m_1,a_2,\bar b)\in \tp(d/M_1 a_2 \bar b)$. By $\bullet_2$, there is $d'\in M_1$ such that $(m_1,d')\equiv_{M} (m_1,d)$. As $\tp(a_2,\bar b/N)$ is $M$-invariant, we have $d'\models \zeta(y;m_1,a_2,\bar b)$ as required. It follows that one of $\theta_0(d)$ or $\theta_1(d)$ must hold. But since $\phi(a_1,a_2;d)$ holds, $\theta_0(d)$ cannot hold by $(*)$. So the claim is proved.

\smallskip \noindent
\underline{Claim 3}: $\tp(\bar b/M_1 a_2 d)$ is finitely satisfiable in $M$.

Proof: Let $\phi(\bar x;m_1,a_2,d)\in \tp(\bar b/M_1a_2d)$, $m_1\in M_1$. By $\bullet_2$, there is $d'\in M_1$ such that $m_1\hat{~}d'\equiv_M m_1\hat{~}d$. As $\tp(a_1a_2\bar b/N)$ is $M$-invariant, we also have $\phi(\bar x;m_1,a_2,d')\in \tp(\bar b/M_1a_2)$. By construction, and using that $\tp(a_2/N)$ commutes with $\tilde q$ over $M_1$, $\bar b$ realizes a Morley sequence of $\tilde q$ over $M_1a_2$. Hence $\tp(\bar b/M_1a_2)$ is finitely satisfiable in $M$. We thus get some $\bar b'\in M$ such that $\phi(\bar b';m_1,a_2,d')$ holds, but then so does $\phi(\bar b';m_1,a_2,d)$ since $m_1\hat{~}d$ and $m_1\hat{~}d'$ have the same type over $Ma_2$. This proves the claim.
\medskip

Let $(e_i,g_i:i<n)\in M_1$ be a sufficiently long strict Morley sequence over $M$ with $(e_0,g_0)=(e,g)$. By $M$-invariance of $\tp(a_1a_2\bar b/N)$, $\psi(a_1,a_2;\bar b,e_i)\wedge \psi_*(a_1,a_2;g_i)$ holds for all $i<n$.

By Claim 3, there is $\bar b'\in M$ satisfying:

-- $(\exists x)\bigwedge_{i<n} \psi(x,a_2;\bar b',e_i) \wedge \psi_*(x,a_2;g_i)$;

-- $(\forall x,y) \psi(x,a_2;\bar b',e)\wedge \theta_1(a_2,\bar b',e;y)\impl \phi(x,a_2;y)$;

-- $\theta_1(a_2,\bar b',e;d)$.

By Fact \ref{fact_betafork} and having taken $n$ large enough, the first point implies that $\psi(x,a_2;\bar b',e)\wedge \psi_*(x,a_2;g)$ does not $a_2$-fork over $M$. So we can find $a'_1$ realising that formula such that $\tp(a'_1,a_2/N)$ is $M$-invariant.

Set $\Gamma'_{L}=\Gamma_{L} \cup \{\phi(x_1,x_2;y)\}$ and $\Gamma'=\Gamma\cup \{(\psi(x_1,x_2),\theta_1(x_2;d))\}$, then the pair $a'_1\hat{~}a_2$ is $\Gamma'$-good.

%
%
%
%
%
%

\medskip
Now to prove the proposition, we iterate the above procedure using $a'_1\hat{~}a_2$ instead of $a_1,a_2$ and $\Gamma'$ instead of $\Gamma$. If we stop at some finite stage, we have what we want. If not, then we have defined a sequence $a^{k}_1$, $k<\omega$ of tuples and increasing sets $\Gamma_k$. Let $a'_1\in \monster$ be such that $\tp(a'_1,a_2/N)$ is an accumulation point of $\tp(a^{k}_1,a_2/N)$. Then $a'_1\hat{~}a_2$ is $\Gamma$-good, where $\Gamma=\bigcup_{k<\omega} \Gamma_k$ and its type over $N$ is $M$-invariant. Hence by the first claim, we are done.
\end{proof}

\begin{rem}
Why do we bother with $M_1$? The problem is that to make sure that $\Gamma$ increases throughout the construction, we need it to remain finite. So we can only deal with a countable set of parameters. This is the role of $M_1$: it controls a priori the parameters from $\psi$ and $\theta_1$.
\end{rem}

\subsection{Directionality}

Recall that an NIP theory $T$ is of \emph{small directionality}, if given a model $M$ and $p\in S(M)$, then for any finite set $\Delta$ of formulas, the global coheirs of $p$ determine only finitely many $\Delta$-types (and thus $p$ has at most $2^{|T|}$ coheirs). It is of \emph{medium directionality} if it is not of small directionality and if the global coheirs of every such $p$ determine at most $|M|$ $\Delta$-types (and thus $p$ has at most $|M|^{|T|}$ coheirs).

Those notions are defined and investigated by Kaplan and Shelah in \cite{Sh946}.
%
%
%
%
%
%
%
%
%
\begin{thm}\label{th_medium}
If $T$ is countable, dp-minimal and of small or medium directionality, then Conjecture \ref{conj_main} holds.
\end{thm}
\begin{proof}
Let $M\models T$ be countable, $M\prec^+ N$ and $\phi(x_1,\ldots,x_n;d)\in L(N)$ non-forking over $M$. By the remark after Lemma \ref{lem_comloc}, we need to prove that there is $\bar a=(a_1,\ldots,a_n)\in \phi(\monster;d)$, $\tp(\bar a/N)$ does not fork over $M$ and commutes over $M$ with every coheir of $\tp(d/M)$.

Let $Q\subset S_y(\monster)$ be a countable set of coheirs of $q=\tp(d/M)$ such that for every finite $\Delta$, and any coheir $\tilde q$ of $q$, there is $s\in Q$ such that $\tilde q$ and $s$ have the same restriction to instances of formulas in $\Delta$. Let $\tilde s=\bigotimes_{s\in Q} s$ (the product being taken in any order). Let also $\tilde q$ be a strictly non-forking coheir of $q$ (which exists by \cite[Proposition 3.7 (1)]{CherKapl}). Finally, let $\bar e$ in $N$ realize $\tilde s \otimes \tilde q^{(\omega)}$.

We can find a countable model $M_1$ such that $M\prec M_1 \prec N$, $\tp(\bar e/M_1)=\tilde s\otimes \tilde q^{(\omega)}|_{M_1}$ and $\bullet_2$, $\bullet_3$ are satisfied with $\bar e$ instead of $d$. Without loss, assume that $d\in \bar e$ and call $\bar d=(d_i:i<\omega)$ the realisation of $\tilde q^{(\omega)}$ in $\bar e$.

We now build by induction on $k\leq n$ tuples $(a_1^k,\ldots,a_n^k)\models \phi(x_1,\ldots,x_n;d)$ such that $\tp(a_1^k,\ldots,a_n^k/N)$ is $M$-invariant and $\tp(a_1^k,\ldots,a_k^k/N)$ commutes with every coheir of $\tp(\bar e/M)$ extending $\tp(\bar e/M_1)$.

As $\phi(x_1,\ldots,x_n;d)$ does not fork over $M$, there is $(a_1^0,\ldots,a_n^0)\models \phi(x_1,\ldots,x_n;d)$ such that $\tp(a_1^0,\ldots,a_n^0/N)$ is $M$-invariant.

Assume that for some $k<n$, we have found $(a_1^k,\ldots,a_n^k)$. For $l\leq k$, set $a_{l}^{k+1}=a_l^k$. Fix some $m>\alt(\phi)$ and define 
$$\phi_k(x_1,\ldots,x_{k+1};\bar d) = (\exists x_{k+2},\ldots,x_n) \bigwedge_{i<m} \phi(x_1,x_2,\ldots,x_n;d_i).$$

Note that by $M$-invariance, $(a^k_1,\ldots,a^k_{k+1})\models \phi_k$. By Proposition \ref{prop_mainprop}, we may find $a^{k+1}_{k+1}\in \monster$ such that $(a^k_1,\ldots,a^k_{k},a^{k+1}_{k+1})\models \phi_k$ and $\tp(a^k_1,\ldots,a^k_{k},a^{k+1}_{k+1}/N)$ is $M$-invariant and commutes with every coheir of $\tp(\bar e/M)$ extending $\tp(\bar e/M_1)$. Next, by the properties of strict non-forking sequences, we know that the formula $\phi(a^k_1,\ldots,a^k_{k},a^{k+1}_{k+1},x_{k+2},\ldots,x_n;d)$ does not $(a^k_1,\ldots,a^k_{k},a^{k+1}_{k+1})$-fork over $M$. Hence we may find $a^{k+1}_{k+2},\ldots,a^{k+1}_n \in \monster$ such that $\phi(a^{k+1}_1,\ldots,a^{k+1}_n;\bar d)$ holds and $\tp(a^{k+1}_1,\ldots,a^{k+1}_n/N)$ is $M$-invariant. This finishes the induction.

Let $\bar a=(a^n_1,\ldots,a^n_n)$.

\smallskip \noindent
\underline{Claim}: $p=\tp(\bar a/N)$ commutes over $M$ with every coheir of $\tp(d/M)$.

Proof: By construction, $p$ commutes over $M_1$ with $\tilde s\otimes \tilde q^{(\omega)}$. In particular, $p$ commutes over $M$ with any $s\in Q$. Let $\tilde q$ be any coheir of $\tp(d/M)$ and $\psi(x;y)\in L(M)$ a formula. Assume that $p\vdash \psi(x;d)$. Let $s\in Q$ be such that $s$ and $\tilde q$ have the same restriction to instances of $\psi$. Then $p$ commutes with $s$, hence $s\vdash \psi(a;y)$, so $\tilde q\vdash \psi(a;y)$. As this is true for all $\psi$, $p$ commutes with $q$ over $M$.

This finishes the proof.
%
%
\end{proof}

At this point one would hope that every dp-minimal theory is of small or medium directionality, but unfortunately this is not true. In fact RCF has large directionality (see \cite{Sh946}).

\section{Dp-rank}\label{sec_dprank}

In this section, we will always distinguish points $a,b,\ldots$ from tuples $\bar a,\bar b,\ldots$. We \emph{do not} work in $T^{eq}$ and in fact most of our results do not carry through to imaginaries.

\smallskip
The reader should have in mind the definition and basic properties of dp-rank as recalled in Section \ref{sec_dprankintro}.

\smallskip\noindent
\textbf{Assumption:} From now on, $T$ is a dp-minimal (one-sorted) theory.

\subsection{Additivity and $\acl$-dimension}

To begin with, we characterise when dp-rank is additive, {\it i.e.}, when $\dpr(\bar a,\bar b/A)=\dpr(\bar a/A\bar b )+\dpr(\bar b/A)$. By an immediate induction, this is equivalent to $\dpr(a,\bar b/A)=\dpr(a/A\bar b)+\dpr(\bar b/A)$ for all $a,\bar b,A$.

We say that $\acl$ satisfies exchange if for any base $A$ and any two points $a,b$, we have $b\in \acl(Aa)\setminus \acl(A) \Longrightarrow a\in \acl(Ab)$. In this case, $\acl$ defines a pregeometry and gives rise to a dimension in the usual way.

\begin{obs}[$T$ is dp-minimal]
Assume that dp-rank is additive, then $\acl$ satisfies exchange.
\end{obs}
\begin{proof}
Note that for a point $a$, we have $\dpr(a/A)=1\iff a\notin \acl(A)$.

Let $A,a,b$ such that $b\in \acl(Aa)\setminus \acl(A)$. Necessarily, $a\notin \acl(A)$. Then $\dpr(ab/A)=\dpr(a/A)+\dpr(b/aA)=1$. We also have $\dpr(ab/A)=\dpr(b/A)+\dpr(a/bA)=1+\dpr(a/bA)$. Hence $\dpr(a/bA)=0$ and $a\in \acl(bA)$.
\end{proof}

We now show the converse.

\begin{prop}[$T$ is dp-minimal]
Assume that $\acl$ satisfies exchange, then, over any base $A$:

$\bullet_0$ dp-rank is additive;

$\bullet_1$ dp-rank coincides with $\acl$-dimension;

$\bullet_2$ if $a_1,\ldots,a_n$ are $\acl$-independent, then we can find mutually indiscernible non-constant sequences $I_1,\ldots,I_n$ such that $I_k$ begins with $a_k$.
\end{prop}
\begin{proof}
It is enough to prove $\bullet_2$ (because it shows that dp-rank is bounded below by $\acl$-dimension, and the opposite inequality follows from sub-additivity). We show it by induction on $n$. For $n=1$, it is clear.

Assume it for $n$ and let $a_1,\ldots, a_{n+1}$ be $\acl$-independent over some base $A$. Then $a_1,\ldots,a_n$ are $\acl$-independent over $Aa_{n+1}$, hence by induction hypothesis, we can find non-constant sequences $I_1,\ldots,I_n$ mutually indiscernible over $Aa_{n+1}$ and $I_k$ starts with $a_k$.

\smallskip
\underline{Claim}: $a_{n+1}\notin \acl(A,I_1,\ldots,I_n)$.

If this is not true, then by exchange, there are some finite set $B\subset A\cup I_1\cup \cdots \cup I_n$, $k\leq n$ and $b\in I_k$ such that $b\notin \acl(B)$, but $b\in \acl(Ba_{n+1})$. Now increase all the sequences to be of order type $\mathbb Q$. We see that there are infinitely many points having the same type as $b$ over $Ba_{n+1}$. A contradiction.

\smallskip
We can therefore find a sequence $I_{n+1}$ which is indiscernible over $AI_1\ldots I_n$ and begins with $a_{n+1}$. By Ramsey's theorem, we can make the sequences $I_1,\ldots,I_n$ mutually indiscernible over $AI_{n+1}$ which gives what we want.
\end{proof}

Here is a simple example where dp-rank is not equal to $\acl$-dimension. Take $L=\{R\}$ and $T$ says that $R$ defines a graph-theoretic tree (a graph with no cycle) where each node has infinite degree. Given two points $a,b$, either $a$ and $b$ are in different connected components, or there is a unique path between $a$ and $b$. The length of this path is called the distance between $a$ and $b$. In this latter case, all the elements in the path are in $\acl(a,b)$. Furthermore $\acl(a)=\{a\}$ for any $a$. Take $aRb$, then $a,b$ are $\acl$-independent, however one can check that $\dpr(a,b)=1$.

Note that $T$ is $\omega$-stable: over a model $M$ there is a unique type of Morley rank $\omega$ at infinite distance of all $a\in M$, and for each $n<\omega$ and $a\in M$, there is a unique type of Morley rank $n$ of an element at distance $n$ from $a$ and at distance $>n$ from all other points of $M$. Also $T$ is dp-minimal: this last fact can be checked directly, or can be seen to follow from \cite[Theorem 4.7]{dpmin} which says that any order-theoretic tree is dp-minimal.

\subsection{Strong witnesses of dp-rank}

The main technical result of this section is a generalisation of $\bullet_2$ above which does not involve $\acl$-independence.

Let $A$ be a set of parameters and $I=(c_i:i\in \mathcal I)$ an $A$-indiscernible sequence indexed by a dense order $\mathcal I$. We will say that a tuple $\bar a$ \emph{breaks} $I$ over $A$ if for some $u\in \mathcal I$, $\bar a$ breaks $I$ at $c_u$ which means that there are $v<u<w$ and a formula $\phi(\bar a,y)$ with parameters in $A$ such that either:

$\cdot$ for all $i\in (v,w)$, $\phi(\bar a,c_i)$ holds if and only if $i=u$

\noindent
or

$\cdot$ for all $i\in (v,w)\setminus \{u\}$, $\phi(\bar a,c_i)$ holds if and only if $i>u$.

In particular, $I$ is not indiscernible over $A\bar a$.

Let $I=(c_i:i\in \mathbb Q)$ be $A$-indiscernible and $\bar a$ a tuple. Assume that $\bar a$ breaks the sequence $I$ at $k$ different places, then $\dpr(\bar a/A)\geq k$. Indeed, assume for example that $\bar a$ breaks $I$ at $c_0,\ldots,c_{k-1}$. We can divide $I$ into $k$ sequences $I_0=(c_i:i<1/2)$, $I_1=(c_i:1/2<i<3/2)$, $\ldots$, $I_{k-1}=(c_i:(2k-1)/2<i)$. The sequences $(I_i:i<k)$ are mutually indiscernible over $A$ and none remains indiscernible over $A\bar a$.

The following is an easy exercise on indiscernible sequences.

\begin{lemme}\label{lem_witnessbreak}
Let $p\in S(A)$ and $n<\omega$. Assume that $\dpr(p)\geq n$ and let $a\models p$, then there are sequences $I_0,\ldots,I_{n-1}$ mutually indiscernible over $A$, $I_k=(c^k_i:i\in \mathbb Q)$, such that for each $k<n$, $a$ breaks the sequence $I_k$ at $c^k_0$.
\end{lemme}
\begin{proof}
Start with sequences $J_0,\ldots,J_{n-1}$ witnessing that $\dpr(\tp(a/A))\geq n$. So the $J_k$'s are mutually indiscernible over $A$ and none of them remains indiscernible over $Aa$. Without loss, assume that each $J_k$ is indexed by a very saturated dense linear order with no endpoints.

Write $J_0=(c_i:i\in \mathcal J)$. As $J_0$ is not indiscernible over $Aa$, there are $u<v \in \mathcal J$ and a formula $\phi(x;y,m)$ with parameters $m\in A \cup (c_i:i\in \mathcal J \setminus [u,v])$ such that $\models \phi(a;c_u,m)\wedge \neg \phi(a;c_v,m)$. Let $\mathcal J'$ be an interval of $\mathcal J$ containing $u,v$ and disjoint from the indices of elements of $m\cap J_0$. Let $J'_0= (c_i\hat{~}m:i\in \mathcal J')$. By NIP (finite alternation), one can partition $\mathcal J'$ into finitely many convex subsets such that the truth value of $\phi(a;y\hat{~}z)$ is constant on each. It is then easy to extract from $J'_0$ a subsequence $I_0$ indexed by $\mathbb Q$ such that $a$ breaks $I_0$ at some point.

Doing inductively the same for each $J_k$, we obtain what we want.
\end{proof}

\begin{prop}[$T$ is dp-minimal]\label{prop_mainlemma}
Assume that $\dpr(a_0,\ldots,a_{n-1}/A)=r$, then for some indices $i_1,\ldots,i_r \in n$, there are (non-constant) mutually indiscernible sequences (over $A$) $I_1,\ldots,I_r$ starting respectively with $a_{i_1},\ldots,a_{i_r}$.
\end{prop}
\begin{proof}
We prove the result by induction on $n$. Assume that $\dpr(\bar a,b/A)=r$, where $\bar a=(a_0,\ldots,a_{n-1})$ and let $J_1,\ldots,J_r$ be given by Lemma \ref{lem_witnessbreak} applied to $\bar a\hat{~} b$. If $\dpr(\bar a/A)=r$, we conclude by induction. Otherwise, there is some sequence, say $J_1$, which is indiscernible over $A\bar a$.

Write, $J_1=(c_i:i\in \mathbb Q)$ such that $\bar a\hat{~}b$ breaks $J_1$ at $c_0$. Set $b_0=b$ and for every $0<k<\omega$, find some $b_k$ such that $$\tp(\bar a,b_k,(c_{i+k})_{i\in \mathbb Q}/A)=\tp(\bar a,b,(c_i)_{i\in \mathbb Q}/A).$$

In particular $\bar a\hat{~}b_k$ breaks the sequence $J_1$ at $c_k$.

Let $\mathbf b=(b'_i:i<\omega)$ be a sequence indiscernible over $A\bar a$ and realising the EM-type of $(b_k:k<\omega)$ over $A\bar a$.

\smallskip
\underline{Claim}: $\dpr(\bar a/A\mathbf b)\geq r-1$.

Assume that $\dpr(\bar a/A\mathbf b)\leq r-2$. Then we cannot construct $r-1$ sequences as in the statement of the proposition. By compactness, we can find some formula $\phi(x_{\bar a},\bar b)\in \tp(\bar a/A\mathbf b)$ which ensures this (where $\bar b\in \mathbf b$). Then by the induction hypothesis, any $\bar a'$ satisfying $\phi(x_{\bar a},\bar b)$ has dp-rank over $A\bar b$ which is $\leq r-2$. Also, again by compactness, there is some formula $\psi(\bar y)\in \tp(\bar b/A)$ such that the same holds for $\phi(x_{\bar a},\bar b')$ whenever $\bar b'\models \psi(\bar y)$.

By construction of $\mathbf b$, we can find such a $\bar b'$ in the original sequence $(b_k:k<\omega)$. Let $m=|\bar b'|$ and without loss $b\in \bar b'$. So we have $\dpr(\bar a/A\bar b')\leq r-2$. Consider the tuple $\bar b'\hat{~}\bar a$. By sub-additivity its dp-rank over $A$ is $\leq m+r-2$. On the other hand, it breaks the sequence $J_1$ at $m$ different places, and it also breaks each of the sequences $J_2,\ldots,J_r$. Hence $\dpr(\bar b',\bar a/A)\geq m+r-1$. This contradiction proves the claim.

\smallskip
Now: we have $\dpr(\bar a/A\mathbf b)\geq r-1$ and $\mathbf b$ is indiscernible over $A\bar a$. As all the points in the original sequence $(b_k:k<\omega)$ have the same type over $A\bar a$, we may assume that $b$ is in $\mathbf b$. By induction hypothesis, we find sequences $I_1,\ldots,I_{r-1}$ each starting with a point from the tuple $\bar a$ and mutually indiscernible over $A\mathbf b$. Let $I_r$ be indiscernible over $AI_1,\ldots,I_{r-1}\bar a$ and realise the EM-type of $\mathbf b$ over it. As all elements from $\mathbf b$ have the same type over $A\bar a$, we may assume that the sequence $I_r$ starts with $b$. So we are done.
\end{proof}

Note that conversely, if there are sequences $I_1,\ldots,I_r$ as in the statement of the theorem, then we have $\dpr(a_1,\ldots,a_n/A)\geq r$, as witnessed by those sequences.

Part of the following corollary was observed during the proof. The third bullet is immediate, and the other two follow from the proposition by compactness.

\begin{cor}\label{cor_one}
Assume that $T$ is dp-minimal.

$\bullet_1$ If $\dpr(\bar a/A)=r$, then there is $\phi(\bar x)\in \tp(\bar a/A)$ such that $\dpr(\phi(\bar x))=r$.

$\bullet_2$ Let $\phi(x;y)$ be a formula over $A$. Then the set $\{c:\dpr(\phi(x;c))\leq r\}$  is open over $A$.

$\bullet_3$ Assume that $\dpr(a_1,\ldots,a_n/A)=r$, then there are $i_1,\ldots,i_r$ such that $\dpr(a_{i_1},\ldots,a_{i_r}/A)=r$.
\end{cor}

We give examples of theories of finite rank where the first two bullets fail.

Let $L=\{E_n,F_n,c :n<\omega\}$, where $E_n$ and $F_n$ are binary relations and $c$ is a constant symbol. The axioms of $T$ say that $E_n$ and $F_n$ are equivalence relations with only infinite classes. Each $E_{n+1}$ (resp. $F_{n+1}$) refines $E_n$ (resp. $F_n$) and the $F_n$'s are cross-cutting with respect to the $E_n$'s. The theory $T$ admits elimination of quantifiers in the language $L$ and is stable.

Consider the type $\{xE_n c \wedge xF_n c:n<\omega\}$. Then that type has dp-rank 1, it is in fact a minimal type. However any formula in it has dp-rank 2.

To obtain an example where the second bullet fails, we modify slightly the previous one. Consider a two sorted structure $(M,S)$. The first sort $M$ is a model of the previous theory. The second sort $S$ is isomorphic to $(\omega;<)$. There is in addition a binary relation $R(x;s)\subseteq M\times S$ interpreted so that $R(a;k)$ holds if and only if $M\models a E_k c \wedge a F_k c$. It is not too hard to see that this theory is NIP. Now in a saturated model, $\dpr(R(x;s))\leq 1$ holds if and only if $\bigwedge_{n<\omega} s>n$, which is not an open condition.

In fact those two theories can be interpreted in a dp-minimal theory. For the first one, simply take $L_0=\{e_n(x;y)\}$ and $M_0$ an $L_0$-structure where each $e_n$ defines an equivalence relation with only infinite classes and $e_{n+1}$ refines $e_n$. Then the first structure $M$ is definable in $M_0^2$.

To deal with the second one, add a sort $(\omega,<)$ to $M_0$ and a binary predicate $r(x;s)$ interpreted so that $r(M_0,n)$ is some $e_n$-class and the sequence $(r(M_0,n):n<\omega)$ is decreasing with non empty intersection. Call $M_1$ the resulting structure. The structure $(M,S)$ above can be defined in $M_1^2$.

We sketch a proof that $M_1$ is dp-minimal. To obtain quantifier elimination, first add to $(\omega,<)$ predicates $d_n(x,y)$ saying that $x$ and $y$ are at distance $n$, then add a function symbol $f$ from the main sort to the order sort defined so that $f(x)$ is maximal such that $R(x;f(x))$ holds (or equal to 0 if no such value exists). Now consider a point $x$ and two mutually indiscernible sequences $I=(\bar a_i:i<\omega)$ and $J=(\bar b_i:i<\omega)$. Without loss (using $f$), everything lives in the main sort. As we have EQ in a binary language, we may assume that $\bar a_i=a_i$ and $\bar b_i=b_i$ are singletons. Also, we may assume that none of the $a_i$'s or $b_i$'s is equal to $x$. Then the sequence $I$ can fail to be indiscernible over $x$ for one of two reasons: either for some $n,i<\omega$, $\neg (a_0 e_n a_1)$ and $x e_n a_i$ \underline{or} $f(a_i)\neq 0$ for all $i$, $f(x)\neq 0$ and the sequence $(f(a_i):i<\omega)$ is not indiscernible over $f(x)$. The same goes for $J$. It is then routine to check that none of the four senari can happen.

\subsection{Characterising dp-rank}

We now aim at showing $\oplus_3$ from Theorem \ref{th_summary} which says that a set of maximal dp-rank is a product of 1-dimensional sets minus a hypersurface. To simplify the exposition, we first deal with an easy case, when the structure is linearly ordered.

\subsection{The linearly ordered case}\label{sec_linear}

In this section $(M,\leq)$ is a dp-minimal densely ordered structure and assume that any definable set in dimension 1 is the union of an open set and finitely many points. By Theorem \cite[Theorem 3.6]{dpmin}, this holds in particular if $M$ expands a divisible ordered group.


An \emph{open box} of $\monster^n$ is a definable set of the form $\{(x_0,\ldots,x_{n-1}): a_k<x_k<b_k\}$ for some $a_k,b_k\in \monster$, $a_k<b_k$. Note that an open box of $\monster^n$ has dp-rank $n$. A tuple $\bar a$ lies in the interior of a set $\phi(\bar x)$ if it lies in some open box included in $\phi(\bar x)$. This is a definable condition, hence the interior of a definable set is again definable.

\begin{prop}
With the assumptions above, let $\phi(\bar x)\in L(A)$ be a definable subset of $M^n$ which is of dp-rank $n$. Then $\phi(\bar x)$ has non-empty interior (that is, contains some $n$-dimensional open box).
\end{prop}
\begin{proof}
We prove the result by induction on $n$. By assumption, the result is true for $n=1$. Assume that we know it for $n$. Consider a formula $\phi(x,\bar y)$, $|\bar y|=n$, of dp-rank $n+1$. By Proposition \ref{prop_mainlemma}, we can find two non-constant mutually indiscernible sequences $(a_i:i<\omega)$ and $(\bar b_j:j<\omega)$ such that $\phi(a_i,\bar b_j)$ holds for all $i,j$. Pick some $i<\omega$ and consider the set $\phi(a_i,\bar y)$. This is a subset of $M^n$, which contains each $\bar b_j$. Let $\psi(a_i,\bar y)$ be its interior. Then by induction hypothesis, the set $\phi(a_i,\bar y)\wedge \neg \psi(a_i,\bar y)$ has dp-rank $<n$. Hence all the $\bar b_j$'s lie in $\psi(a_i,\bar y)$.

Fix $j<\omega$. Then by compactness, there is some open box $\theta_j(\bar y)$ defined over $\monster$ such that $\bar b_j\models \theta_j(\bar y)$ and $\theta_j(\bar y)$ is included in each $\psi(a_i,\bar y)$. Consider the set of all $a\in \monster$ such that $\phi(a;\monster)$ contains $\theta_j(\monster)$. This is a definable set which contains all the $a_i$'s. In particular, it is infinite and therefore has non-empty interior. Let $\psi_j(x)$ be an open interval in it. Then $\psi_j(x)\wedge \theta_j(\bar y)$ is an open box included in $\phi(x,\bar y)$.
\end{proof}

Using Corollary \ref{cor_one} we conclude that the dp-rank of a set $X$ coincides with the maximal $n$ such that some projection of $X$ to $M^n$ has non-empty interior. One therefore has a nice dimension theory to work with. In fact, one could hope to prove that definable sets in any dimension are tame in some sense: take for example a definable set $X$ in $M^2$. Then we know from the theorem above that $\overline{X}\setminus X$ has dp-rank 1. It should then be possible to show that such a set cannot be too complicated. Since the plane can be linearly ordered by a lexicographic ordering, the results in \cite{dpmin} might be relevant. We will not pursue this.

\begin{cor}
Dp-rank is definable: for every formula $\phi(\bar x;\bar y)$ the set of tuples $\bar b$ such that $\dpr(\phi(\bar x;\bar b))=k$ is a definable set.
\end{cor}
\begin{proof}
One can express in a first order way the fact that a set contains some $k$-dimensional box. Then the result follows from the remark above.
\end{proof}

\subsection{The distal case}

We now generalise the previous result to any dp-minimal theory which has no generically stable types. First recall some definitions: A generically stable type $p$ is a global type which is both definable and finitely satisfiable in some small model $M$. Equivalently, it is an $M$-invariant type whose Morley sequence is totally indiscernible. The restriction of a generically stable type to a subtuple of variables is again generically stable. Hence if there is a non-realised generically stable type of some arity, there is one of arity one. Clearly, if the structure admits a linear order, then there is no non-realised generically stable type, since there is no non-constant totally indiscernible sequence.

An NIP theory is called \emph{distal} (\cite{distal},\cite[Chapter 9]{NIPbook}) if the following property is satisfied: whenever $I_0+I_1+I_2$ is an $A$-indiscernible sequence of tuples, $I_0$ and $I_2$ are infinite and $\bar b$ is a tuple, if $I_0+I_2$ is indiscernible over $A\bar b$, then $I_0+I_1+I_2$ is indiscernible over $A\bar b$. It is easy to see from the definition that a distal theory cannot have a totally indiscernible, non-constant, sequence (take $\bar b$ to be a tuple in $I_1$). In particular, it cannot have a non-realised generically stable type. It is shown in \cite[Corollary 2.30]{distal}, also \cite[Corollary 9.19]{NIPbook}, that the converse holds for dp-minimal theories: A dp-minimal theory $T$ is distal if and only if it has no non-realised generically stable types. In particular, any linearly ordered dp-minimal theory is distal (hence any o-minimal or weakly-o-minimal theory) and also $Th(\mathbb Q_p)$ is distal.

\begin{thm}\label{th_rectangle}
Assume that $T$ is distal (and dp-minimal). Let $\phi(\bar x)\in L(A)$ have dp-rank $n=|\bar x|$ over $A$. Then $\phi(\bar x)$ contains a product $\theta_0(x_0)\wedge \cdots \wedge \theta_{n-1}(x_{n-1})$ where each $\theta_k(x_k)\in L(\monster)$ defines an infinite set.
\end{thm}
\begin{proof}
We prove the result by induction on $n$. For $n=1$ it is clear. Assume that we know it for $n$ and let $\phi(x,\bar y)$ have dp-rank $n+1$ over $A$, $|\bar y|=n$. Take $(a_0,b_1,\ldots,b_n)$ satisfying $\phi(x,\bar y)$ such that $\dpr(a_0,\bar b/A)=n+1$. By Proposition \ref{prop_mainlemma}, we can find sequences $(I,J_1,\ldots, J_n)$ mutually indiscernible over $A$ such that $I$ starts with $a_0$ and $J_k$ starts with $b_k$. Without loss, $I$ is ordered by $\omega+\omega$ and we can write $I=(a_i:i<\omega+\omega)$. By indiscernability, $\phi(a_i,\bar b)$ holds for all $i$. Also, $\dpr(\bar b/IA)=n$ as witnessed by the sequences $(J_1,\ldots,J_n)$.

Let $I_0=(a_i:i<\omega)$ and $I_1=(a_i:\omega< i<\omega+\omega)$. By distality of $T$, for any $a$ such that $I_0+(a)+I_1$ is indiscernible over $A$, $\phi(a,\bar b)$ holds. Therefore by compactness, there is a formula $\theta_0(x)\in \tp(a_\omega/IA)$ such that $\models \theta_0(x)\rightarrow \phi(x,\bar b)$. Let $\psi(\bar y)=\forall x(\theta_0(x)\rightarrow \phi(x,\bar y))$. Then $\psi(\bar y)$ is a formula over $IA$ satisfied by $\bar b$. As $\dprk(\bar b/IA)=n$, also $\dpr(\psi(\bar y))=n$. We now apply the induction hypothesis to $\psi(\bar y)$ to obtain $\theta_1(x_1),\ldots,\theta_n(x_n)$.
\end{proof}

Note that conversely, if a set contains such a conjunction, then it has dp-rank $n$.


\subsection{The general case}

We now deal with the general case.
%
%

\begin{defi}
A definable set $\phi(x_0,\ldots,x_{n-1})$ is called a \emph{hypersurface} if for every $a_0$, the set $\phi(a_0,x_1,\ldots,x_{n-1})$ has dp-rank $<n-1$.
\end{defi}

By convention, the formula $x_0\neq x_0$ is the only hypersurface in dimension 1.

Note that a hypersurface $\phi(x_0,\ldots,x_{n-1})$ has dp-rank $<n$.

\begin{thm}[$T$ is dp-minimal]\label{th_rectanglend}
Let $\phi(\bar x)\in L(A)$ have dp-rank $n=|\bar x|$. Then there are formulas $\theta_k(x_k)\in L(\monster)$, $k=0,\ldots,n-1$ defining infinite sets and a hypersurface $\psi(\bar x)\in L(\monster)$ such that $\bigwedge_{k<n} \theta_k(x_k) \rightarrow (\phi(\bar x)\vee \psi(\bar x))$.
\end{thm}
\begin{proof}
We prove the result by induction on $n$. It is clear for $n=1$. Assume that we know it for $n$ and let $\phi(x_0,\bar x)\in L(A)$ have dp-rank $n+1$ over $A$. As in the proof of Theorem \ref{th_rectangle}, we can find a sequence $I=(a_i:i<\omega)$ and $\bar b$ such that $I$ is indiscernible over $A\bar b$, $\dprk(\bar b/IA)=n$ and $\phi(a_i,\bar b)$ holds for all $i$.

By NIP, we can build a maximal sequence $(a'_0,\ldots,a'_{l-1})$, such that:

$\cdot$ $I':=I'_0+(a_i:i\geq l)$ is indiscernible over $A$, where the sequence $I'_0=(a_0,a'_0,a_1,a'_1,\ldots,a_{l-1},a'_{l-1})$;

$\cdot$ $\neg \phi(a'_k,\bar b)$ holds for all $k<l$;

$\cdot$ $\dprk(\bar b/IA+\{a'_k:k<l\})=n$.

Take $a_*$ such that the sequence $I'_0+(a_l,a_*)+(a_i:i>l)$ is indiscernible over $A$. Let $q(\bar x)=\tp(\bar b/AI')$. By maximality of the sequence $(a'_k:k<l)$ the partial type $q(\bar x)\wedge \neg \phi(a_*,\bar x)$ has dp-rank $<n$. By continuity (Corollary \ref{cor_one}, $\bullet_1$), there is a formula $\zeta(\bar x)\in q(\bar x)$ such that $\dprk(\zeta(\bar x)\wedge \neg\phi(a_*,\bar x))<n$. By Corollary \ref{cor_one}, $\bullet_2$, there is some formula $\theta_0(x_0)\in \tp(a_*/AI')$ such that $\dprk(\zeta(\bar x)\wedge \neg \phi(a',\bar x))<n$ for all $a'$ satisfying $\theta_0(x_0)$. Let $\psi_0(x_0,\bar x)=\theta_0(x_0)\wedge \zeta(\bar x)\wedge \neg\phi(x_0,\bar x)$. By construction $\psi_0(x_0,\bar x)$ is a hypersurface.

The induction hypothesis applied to $\zeta(\bar x)$ gives $\theta_1(x_1),\ldots,\theta_n(x_n)$ and a hypersurface $\psi'(\bar x)$. Then define $\psi(x_0,\bar x)=\psi_0(x_0,\bar x)\vee \psi'(\bar x)$.

By unwinding the definitions, one sees that the formulas $\theta_0(x_0),\ldots, \theta_n(x_n)$ and $\psi(x_0,\bar x)$ have the required properties.
\end{proof}

Notice that conversely, if we can find such $\theta_k(x_k)$ and hypersurface $\psi(\bar x)$, then the set $\phi(\bar x)$ has dp-rank $n$.

As already mentioned in the introduction, the hypersurface is necessary: take for example the formula $x_0\neq x_1$ in the theory of equality. It has dp-rank 2, but does not contain a product $\theta_0(x_0)\wedge \theta_1(x_1)$ of infinite 1-dimensional sets.

%

\begin{rem}
In many cases, I expect that the formula $\psi(\bar x;\bar y)$ defining the hypersurface can be chosen to be stable. However, I do not know a sufficient condition that would ensure that.
\end{rem}

\begin{cor}
Assume that $T$ is dp-minimal and eliminates $\exists^{\infty}$. Let $\phi(\bar x;\bar y)$ be a formula and $k<\omega$. Then the set of $\bar b$'s such that $\dprk(\phi(\bar x;\bar b))=k$ is definable.
\end{cor}
\begin{proof}
We show this by induction on $n=|\bar x|$. By Corollary \ref{cor_one}, it is enough to treat the case $k=n$.

A formula in dimension 1 has dp-rank 1 if and only if it is infinite. Thus the case $n=1$ follows from elimination of $\exists^{\infty}$. Assume that we have established the result for $n$ and let $\phi(x_0,\bar x;\bar y)$ be a formula, $\bar x=(x_1,\ldots,x_n)$. Let $\bar b$ be such that $\dpr(\phi(x_0,\bar x;\bar b))=n$. Then Theorem \ref{th_rectanglend} gives us formulas $\theta_k(x_k;\bar e)$, $k\leq n$ and a hypersurface $\psi(x_0,\bar x;\bar e)$, where we have made the parameters $\bar e$ appear. By induction hypothesis and the $n=1$ case, the fact that $\theta_k(x_k;\bar e)$ is infinite for all $k$ and the fact that $\psi(x_0,\bar x;\bar e)$ is a hypersurface are expressible by a formula $\zeta(\bar e)$. (This is where it is important that $\psi(x_0,\bar x;\bar e)$ is a hypersurface and not merely a set of dp-rank $<n$.) It follows that for any $\bar b'$ such that
$$\models\exists \bar z\left [\zeta(\bar z)\wedge \left (\bigwedge_{k\leq n} \theta_k(x_k;\bar z) \impl (\phi(x_0,\bar x;\bar b') \vee \psi(x_0,\bar x;\bar z))\right )\right],$$
the set $\phi(x_0,\bar x;\bar b')$ has dp-rank $n$. Therefore the condition ``$\dpr(\phi(x_0,\bar x;\bar b))=n$" is an open condition in $\bar b$. However we know from Corollary \ref{cor_one}, $\bullet_2$, that it is also closed. Hence it is definable.
\end{proof}

\begin{problem}
Study to what extent some of those results can be generalised (in a weaker form) to arbitrary NIP theories of finite rank.
\end{problem}

\bibliographystyle{asl}
\bibliography{InvTypes}

\end{document}